\documentclass[12pt,reqno,english,a4]{amsart}
\usepackage{amsmath,amsthm,amsfonts,amssymb,amscd}
\usepackage[latin1]{inputenc}
\usepackage{psfrag}
\usepackage{epsfig}
\usepackage{a4wide}

\usepackage{hyperref} 
\theoremstyle{plain}
\newtheorem{theorem}{Theorem}
\newtheorem{proposition}{Proposition}
\newtheorem{lemma}{Lemma}

\newtheorem{example}{Example}
\newtheorem{remark}{Remark}
\newtheorem{corollary}{Corollary}

\numberwithin{equation}{section}

\begin{document}

\title[Foliations by planes]
      {A note on open 3-manifolds supporting foliations by planes}

\author{Carlos Biasi}\thanks{The first author was supported by FAPESP of Brazil
Grant 98/13400-5}

\author{Carlos Maquera}\thanks{The second author was supported by FAPESP of Brazil
Grant 02/09425-0}

\keywords{Foliation by planes, Open manifolds, Fundamental group,
Free group}

\subjclass[2000]{37C85, 57R30}
\date{\today}

\address{Carlos Biasi and Carlos Maquera\\
 Universidade de S{\~a}o Paulo - S{\~a}o
Carlos \\Instituto de Ci{\^e}ncias
Matem{\'a}ticas e de Computa\c{c}{\~a}o\\
Departamento de Matemática\\
Av. do Trabalhador S{\~a}o-Carlense 400 \\
13560-970 S{\~a}o Carlos, SP\\
Brazil}
 \email{biasi@icmc.usp.br}
 \email{cmaquera@icmc.usp.br}
\maketitle

 \begin{abstract}
 We show that if $N$, an open connected $n$-manifold with finitely generated
 fundamental group, is $C^{2}$ foliated by closed planes, then $\pi_{1}(N)$ is a free
 group.
 This implies that if $\pi_{1}(N)$ has an Abelian subgroup of rank greater than
 one, then $\mathcal{F}$ has at least a non closed leaf. Next, we
 show that if $N$ is three dimensional with fundamental group
 abelian of rank greater than one, then $N$ is homeomorphic to
 $\mathbb{T}^2\times \mathbb{R}.$ Furthermore, in this case we
 give a complete description of the foliation.
 \end{abstract}

 \medskip
 \medskip

 \thispagestyle{empty}

 \section{Introduction}
 Two foliated manifolds $(N_{1},\mathcal{F}_{1})$
 and $(N_{2},\mathcal{F}_{2})$ are said to be $C^{r}$ \textit{conjugate}
 if there exists a $C^{r}$ homeomorphism $h:N_{1}\to N_{2}$ that takes leaves of
 $\mathcal{F}_{1}$ onto leaves of $\mathcal{F}_{2}.$
 Let $\mathcal{F}$ be a $C^{2}$ foliation defined on an $n$-manifold
 $N$. If all the leaves are diffeomorphic to $\mathbb{R}^{n-1}$ the
 foliation $\mathcal{F}$ is called a \textit{foliation by planes}.
 $\mathcal{F}$ is called a \textit{Reeb foliation} of $N$ if both $N$ and
 $\mathcal{F}$ are orientable, each leaf of $\mathcal{F}$ in the interior of $N$
 is homeomorphic to $\mathbb{R}^{n-1}$ and if
 $\partial N \neq \emptyset$, then each component of $\partial N$ is a leaf of $\mathcal{F}$
 is homeomorphic to the torus $\mathbb{T}^{n-1}$. Note that, each transversally orientable foliation
 by planes $\mathcal{F}$ on an orientable open $n$-manifold is a
 Reeb foliation.

 Rosenberg--Roussarie in
 \cite{rosen-rossa} have proved that the only compact connected 3-manifolds which
 admit Reeb foliations are $\mathbb{T}^3$, $D^2 \times \mathbb{S}^1$ and $\mathbb{T}^2
 \times [0,1]$. Novikov \cite{Novikov} has proved that all Reeb foliations of
 $D^2 \times \mathbb{S}^1$ are topologically equivalent to the Reeb component.
 Chatelet-Rosenberg \cite {Cha-Rosen} and Rosenberg-Roussarie \cite{rosen-rossa1}
 have classified Reeb foliations of class $C^2$ of $\mathbb{T}^2 \times [0,1]$ and  $\mathbb{T}^3$, respectively.
 In both cases, the authors prove, between other things, that
 all the leaves of $\mathcal{F}$ are dense (in particular there is no closed
 leaf). Let us observe that in these two cases the fundamental group
 of $N$ is abelian of rank grater than one.
 In higher dimension, that is when $N$ is closed and $n$ dimensional, Rosenberg \cite{rosenberg2} proved
 that $\pi_1(N)$ is free abelian. Recently, \'Alvarez L\'opez-Arraut-Biasi \cite{ArrBia} proved that if
 $N\setminus \{p_1,\dots,p_k\}$ is foliated by closed planes, then
 $N=S^n$ (the $n$-sphere) and $k=1.$

 Palmeira \cite{pal} studied transversally orientable $C^{2}$ foliations
 by closed planes on open $n$-manifolds $N$, which
 has finitely generated fundamental group, and proved that
 this foliations are $C^{2}$ conjugates to the product
 of a foliation on an open surface by $\mathbb{R}^{n-2}$. This implies
 that the fundamental group of $N$ is free. Remember that any free group is the
 free product $H_1\ast \cdots\ast H_{\ell}$ where
 $H_i$ is isomorphic to $\{0\}$ or $\mathbb{Z}.$

 In this paper we consider foliations by planes on an open and connected
 $n$-manifold $N$. We try to initiate the classification of some $C^2$ foliations by planes in open
 manifolds having at least one leaf not closed.

  Our first result is the following

 \begin{theorem}
 \label{maintheorem}
 Let $N$ be an open connected $n$-manifold with a $C^{2}$ foliation
 by closed planes. If $\pi_{1}(N)$ is finitely generated, then $\pi_{1}(N)$ is free.
 \end{theorem}

 Note that in Theorem \ref{maintheorem} it is not necessary to assume that $N$ is
 orientable and that the foliation is transversally orientable.
 As a consequence of Theorem \ref{maintheorem}, we have

 \begin{corollary}
 \label{cor:one}
  Let $N$ be an open connected $n$-manifold with a $C^{2}$ foliation
 by planes. If $\pi_{1}(N)$ has an Abelian subgroup of rank greater than
 one, then $\mathcal{F}$ has at least a leaf which is not closed.
 \end{corollary}

 Obviously, the reciprocal of this result is not true, see (2) of Example
 \ref{exem:example 2}. However, the previous corollary motivates the
 following natural question:
 ``if $\pi_{1}(N)$ has an Abelian subgroup (or is Abelian) of rank greater than
 one, then what we can say on the set of  non closed leaves?"
 In this direction, when $N$ is three dimensional, we obtain the following
 result:

 \begin{theorem}
 \label{teo:subgrabelian}
 Let $N$ be an open connected orientable $3$-manifold with a $C^{2}$ foliation
 by planes $\mathcal{F}$ and assume that $\pi_{1}(N)$ is finitely
 generated. Then
 \begin{enumerate}
    \item $\pi_1(N)=H_1\ast \cdots \ast H_{\ell},$ where $H_j,\ j=1,\dots, \ell,$
    is a subgroups of $\pi_1(N)$ isomorphic to $\mathbb{Z}$ or $\mathbb{Z}^2.$
    \item For each $j=1,\dots, \ell,$ satisfying that $H_j$ is isomorphic to $\mathbb{Z}^2,$
    there exists an open submanifold $N_j$ of $N$ that is invariant
    by $\mathcal{F}$, such that $\pi_1(N_j)=H_j$ and $(\mathcal{F}|_{N_j}, N_j)$ is
    topologically conjugated to $(\mathcal{F}_0\times \mathbb{R}, \mathbb{T}^2\times \mathbb{R})$,
    where $\mathcal{F}_0$ is the foliation on the 2-torus $\mathbb{T}^2$ which
    is defined by the irrational flow.
 \end{enumerate}
 \end{theorem}

 By using this theorem we obtain.

 \begin{theorem}
 \label{teo:abeliano}
 Let $N$ be an open connected orientable $3$-manifold with a $C^{2}$ foliation
 by  planes $\mathcal{F}$ and assume that $\pi_{1}(N)$ is finitely generated.
 If  $\pi_{1}(N)$ is an abelian group of rank two, then $N$ is homeomorphic to
 $\mathbb{T}^2\times \mathbb{R}$. Moreover, there exists an open submanifold
 $N_0$ which is homeomorphic to $\mathbb{T}^2\times \mathbb{R}$ and invariant by $\mathcal{F}$,
 such that
  \begin{enumerate}
    \item $(\mathcal{F}|_{N_0}, N_0)$ is topologically conjugated
    to $(\mathcal{F}_0\times \mathbb{R}, \mathbb{T}^2\times \mathbb{R})$, where
    $\mathcal{F}_0$ is the foliation on the 2-torus defined by the irrational flow.

    \item Every connected component $B$ of $N\setminus \overline{N_0}$ is foliated
    by closed planes, hence, homeomorphic to $\mathbb{R}^3.$
    Furthermore, each leaf $L$ in $B$ separates $N$ in two connected
    components.
  \end{enumerate}
 \end{theorem}

 We observe that in the hypothesis of Theorem
 \ref{teo:abeliano}, by Proposition \ref{prop:abeliansubgroup},
 we might have supposed that  $\pi_{1}(N)$ has an Abelian subgroup of rank greater
 than one.

 The proof of Theorem \ref{teo:subgrabelian} is typical of the three dimensional case,
 since we use a result of \cite{gabai} on the existence of incompressible torus in
 irreducible $3$-manifolds.

 This paper is organized as follows. In Section 2 we present some
 examples of foliations by planes in open manifolds having leaves
 not closed. In Section 3, by using two remarkable results obtained by Palmeira (Theorem \ref{teo:Palmeira})
 and Stallings (Lemma \ref{lem:sta}), we prove the Theorem
 \ref{maintheorem}. In Section 4 we prove the Theorem 2 and then by using this theorem we show
 Theorem 3.

 \section{Examples}

 Let us give some examples of open manifolds supporting foliations
 by planes having not closed leaves.

 \begin{example}
 \label{exem:example 2}
   \textnormal{
   Let $X$ be the vector field on the 2-torus $\mathbb{T}^{2}$ which is defined by the
   irrational flow. Fix a point $p$ in $\mathbb{T}^2$ and let $f:\mathbb{T}^2\to \mathbb{R}$
    an application such that $f(p)=0$ and $f(x)\neq 0$ for all $x\in \mathbb{T}^2-\{p\}$.
    Let $\mathcal{F}_{0}$ be the one dimensional foliation in $\mathbb{T}^2-\{p\}$ which is defined by
    the vector field $fX.$ We obtain two foliations:
    \begin{itemize}
      \item[(1)] The foliation by planes $\mathcal{F}_1$ on $N=\mathbb{T}^2\times \mathbb{R}^{n-2}$ whose
      leaves are the product of the orbits of $X$ with $\mathbb{R}^{n-2}$.
      \item[(2)] The foliation by planes  $\mathcal{F}_2$ on $N=(\mathbb{T}^{2}-\{p\})\times \mathbb{R}^{n-2},$
      whose leaves are the product of leaves of $\mathcal{F}_{0}$ with $\mathbb{R}^{n-2}.$
    \end{itemize}
    We can observe that every leaf in the two foliations $\mathcal{F}_1$ and $\mathcal{F}_2$ are dense, in
    particular, are not closed. But, in the first case the fundamental group of $N$ is isomorphic
    to $\mathbb{Z}^2$ and, in the second case the fundamental group of $N$ is free
    not abelian with two generators (because $\pi_1(N)=\pi_1(\mathbb{T}^2-\{p\})$).
    }
   \end{example}

   \begin{figure}[ht!]
 \label{fig:Amalgamation}
 \psfrag{Dt}{\footnotesize $D\times \{t\}$}
 \psfrag{TxI}{\footnotesize$\mathbb{T}^2\times(0,1)$}
       \includegraphics[width=7cm,height=7cm,keepaspectratio]{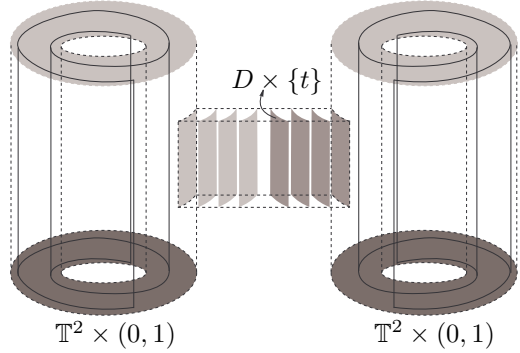}
 \caption{Example 2}
 \end{figure}

\begin{example}
   \textnormal{
   Let $\mathcal{F}_0$ be the foliation on $\mathbb{T}^2 \times [0,1]$ such that
   $T_0=\mathbb{T}^2 \times \{0\}$ and $T_1=\mathbb{T}^2 \times \{1\}$ are compact leaves
   and the restriction $\mathcal{F}_0|_{\mathbb{T}^2 \times (0,1)}$ is
   by planes (in fact, this foliation can be defined by a locally free action of $\mathbb{R}^2$, see
   for example page 2 in \cite{Cha-Rosen}). In this situation, it is well know that, every leaf of
   $\mathcal{F}_0$  in $\mathbb{T}^2 \times (0,1)$ is dense.
   Let $D$ be an open 2-disk in $T_1$ and we consider in $D\times [0,1]$ the
   foliation $\mathcal{F}_d$ whose leaves are $D\times\{t\}$ with $t\in [0,1]$.
   Let $N$ be the open 3-manifolds which is obtained by gluing two copies of
   $[\mathbb{T}^2 \times (0,1)]\cup D$ with $D\times I$ as in Figure \ref{fig:Amalgamation}.
   Let $\mathcal{F}$ be the foliation in $N$ such that
   $\mathcal{F}|_{\mathbb{T}^2 \times (0,1)}=\mathcal{F}_0|_{\mathbb{T}^2 \times (0,1)}$
   and $\mathcal{F}|_{D\times I}=\mathcal{F}_d$.  Then we have the following facts:
   \begin{enumerate}
     \item the fundamental group of $N$ is the free product $\mathbb{Z}^2\ast \mathbb{Z}^2$.
     \item in each copy of $\mathbb{T}^2 \times (0,1)$ we can find two incompressible 2-torus $T_1$ and $T_2$, which are transverse to $\mathcal{F}$, such that the foliations induced by $\mathcal{F}$
         in this torus are topologically equivalent to irrational flow on the 2-torus.
     isotopic to $\mathbb{T}^2 \times \{1/2\}$
     \item for $i=1,2,$ let ${\rm Sat}(T_i)$ be the saturated by $\mathcal{F}$ of $T_i$. Then,
      ${\rm Sat}(T_i)=\mathbb{T}^2\times (0,1)$ and the complement of the closure of ${\rm Sat}(T_1)\cup {\rm
     Sat}(T_2)$ is $D\times(0,1)$ and foliated by closed planes.
   \end{enumerate}
   }
  \end{example}


  \begin{example}
   \textnormal{
   For $k\geq 2,$ let $\mathcal{F}_{0}$ be the $(k-1)$-dimensional foliation by
   planes in the $k$-torus $\mathbb{T}^{k}.$ In $N=\mathbb{T}^{k}\times \mathbb{R}^{n-k},$
   we consider the foliation by planes $\mathcal{F}$ which is the product of
   $\mathcal{F}_{0}$ with $\mathbb{R}^{n-k}.$
   Then:
 \begin{itemize}
    \item \textit{$\pi_{1}(N)=\pi_{1}(\mathbb{T}^{k})=\mathbb{Z}^{k}$ (in particular, Abelian free)}
    \item \textit{All the leaves of $\mathcal{F}$ are denses (in particular, this are
    not closed)}. This, because that every leaf of $\mathcal{F}_{0}$ is dense in $\mathbb{T}^{k}.$
    \item \textit{$N$ is not homeomorphic to the product of a surface by $\mathbb{R}^{n-2}$}
    \item \textit{${\rm rank}(\pi_{1}(N))=k\geq 2$}
 \end{itemize}
   }
  \end{example}

 \section{The fundamental group of an open $n$-manifold foliated by closed planes}
 \subsection{Proof of Theorem \ref{maintheorem}}

 We state the results to be used in the proof of Theorem
 \ref{maintheorem}. We start with the following theorem
 proved by Palmeira in \cite{pal}.

 \begin{theorem}
 \label{teo:Palmeira}
 If $N$ is an orientable open $n$-manifold, $n\geq 3,$ which has
 a finitely generated fundamental group and with a transversely
 orientable $C^2$ foliation $\mathcal{F}$ by closed planes, then
 there exists an orientable surface $\Sigma$ and an orientable
 one-dimensional foliation $\mathcal{F}_0$ of $\Sigma$ such that
 $(N,\mathcal{F})$ is conjugated by a diffeomorphism to
 $(\Sigma \times \mathbb{R}^{n-2},\mathcal{F}_0 \times \mathbb{R}^{n-2}).$
 When $N$ is simply connected it is not necessary to assume either
 that $\mathcal{F}$ is transversely orientable or that
 the leaves are closed, and moreover $\Sigma=\mathbb{R}^2$ in this
 case.
 \end{theorem}

 \begin{corollary}
 $N$ have the same homotopy type of an open surface.
 \end{corollary}

 Theorem \ref{teo:Palmeira} and the following result, which was conjectured by Serre
 \cite{serre} and proved by Stallings \cite{sta}, are crucial to
 prove Theorem \ref{maintheorem}.

 \begin{lemma}[\cite{sta}, {Stallings}]
 \label{lem:sta}
 Let $G$ be a torsion-free finitely generated group. If $H$ is a free subgroup
 of finite index, then $G$ is also free.
 \end{lemma}

  \begin{proposition}
 \label{teo:torsiofree}
 If $N$ is an open connected $n$-manifold foliated by planes, then $\pi_{1}(N)$
 is torsion-free. In particular, if $\pi_{1}(N)$
 is finite, then $N$ is difeomorphic to $\mathbb{R}^{n}.$
 \end{proposition}
 \begin{proof}
 To prove that $\pi_1(N)$ is torsion-free, it is enough to show
 that its only finite subgroup is the trivial one.
 Let $\mathcal{F}$ denote the foliation of $N$ by planes and let
 $H$ be a finite subgroup of $\pi_1(N)$ com $k$ elements.
 Let $\widetilde{N}\to N$ be the universal covering map, and let
 $\widehat{N}\to N$ be the covering map associated to $H$,
 that is, $\pi_1(\widehat{N})=H.$ Let $\widetilde{\mathcal{F}}$
 and $\widehat{\mathcal{F}}$ be the foliations of $\widetilde{N}$
 and $\widehat{N}$ induced by $\mathcal{F}$, both are foliations
 by planes. It follows from the last part of Theorem \ref{teo:Palmeira}
 that $\widetilde{N}=\mathbb{R}^n$. Then the Euler
 characteristics of $\widetilde{N}$ and $N$ satisfy
 $$
 1=\chi(\mathbb{R}^{n})=\chi(\widetilde{N})=k\chi(\widehat{N}), \ \, \chi(\widehat{N})\in\mathbb{Z}.
 $$
 Hence $k=1$ and $H$ is trivial.

 Finally, if $\pi_1(N)$ is finite we obtain that $\pi_1(N)$ is
 trivial, consequently, by the last part of Theorem \ref{teo:Palmeira}
 we have that $N$ is diffeomorphic to $\mathbb{R}^n.$
 \end{proof}

 \begin{proof}[Proof of Theorem \ref{maintheorem}]
  We have the following possibilities

 \begin{itemize}
    \item[(a)] \textit{$N$ is orientable and $\mathcal{F}$ transversally orientable}:
     by Palmeira´s theorem, $\pi_{1}(N)$ is free.
    \item[(b)] \textit{$N$ is orientable and $\mathcal{F}$ is not
    transversally orientable}: we consider $\pi:\widetilde{N}\to N$
    the double covering map that turn $\widetilde{\mathcal{F}}$, the
    lifting to $\widetilde{N}$ of $\mathcal{F}$, transversally
    orientable. Then, by item (a) and Proposition \ref{teo:torsiofree}, we have
    respectively that $\pi_{1}(\widetilde{N})$ is free
    and torsion-free. Finally, since the index of $\pi_{\ast}(\pi_{1}(\widetilde{N}))$
    in $\pi_{1}(N)$ is equal to two, it follows from Lemma
    \ref{lem:sta} that $\pi_{1}(N)$ is free.
    \item[(c)] \textit{$N$ non-orientable}: let $\pi:\widetilde{N}\to N$
    be the double covering map that turn $\widetilde{N}$ orientable. By
    applying the items (a) and (b) to $\widetilde{N}$ and $\widetilde{\mathcal{F}}$, the
    lifting to $\widetilde{N}$ of $\mathcal{F}$, we obtain that $\pi_{1}(\widetilde{N})$
    is free. As above, since the index of $\pi_{\ast}(\pi_{1}(\widetilde{N}))$
    in $\pi_{1}(N)$ is equal to two, it follows from Lemma
    \ref{lem:sta} that $\pi_{1}(N)$ is free.
 \end{itemize}
 \end{proof}

 Let $X$ be a finite set and denote by $\# X$ the cardinality
 of $X$. Let $G$ be a torsion-free finitely generated group. The \textit{rank of}
 $G$ is given by
 $$
 {\rm rank}(G)=\min\{k\in \mathbb{N}; k=\#X\ \textrm{where} \ G\ \textrm{is generated by}\ X\}
 $$

 The following result is equivalent to Corollary \ref{cor:one}.
 \begin{theorem}
 \label{teo:rank}
 Let $\mathcal{F}$ be a foliation by planes of an open connected
 $n$-manifold $N.$ If every leaf of $\mathcal{F}$ is closed, then each
 finitely generate Abelian subgroup $H$ of $\pi_{1}(N)$ has
 rank at most one.
 \end{theorem}
 \begin{proof}
 Let $p:\widetilde{N}\to N$ be the covering such that
 $p_{\ast}(\pi_{1}(\widetilde{N}))=H$. Then by Theorem
 \ref{teo:Palmeira}, $\widetilde{N}$ is diffeomorphic to $\Sigma^{2}\times \mathbb{R}^{n-2},$
  where $\Sigma^{2}$
 is an open surface. Consequently, since $H= p_{\ast}\pi_{1}(\widetilde{N})=p_{\ast}\pi_{1}(\Sigma^{2})$
 is an Abelian group, we have that $\Sigma^{2}$ is either an open cylinder or a plane.
 This completes the proof.
 \end{proof}

 \begin{corollary}
 If each leaf is closed and $\pi_{1}(N)$ is not Abelian,
 then $C=Z(\pi_{1}(N)),$ the center of $\pi_{1}(N),$ is
 trivial.
 \end{corollary}

 \begin{proof}
 Suppose by contradiction that $C\neq\{1\}.$ Since $\pi_{1}(N)$ is not Abelian,
 it follows that $C\subsetneqq \pi_{1}(N).$
 Let $\alpha \in \pi_{1}(N)\setminus C$ and $H$ the Abelian subgroup of
 $\pi_{1}(N)$ which is generated by $\{\alpha\}\cup C.$ Then,
 by Theorem \ref{teo:torsiofree}, ${\rm rank}(H)>1$ contradicting
 Theorem \ref{teo:rank}.
 \end{proof}

 \begin{remark}
 Note that, if $\pi_{1}(N)$ has an Abelian subgroup of rank greater to one,
 or $C\neq \{1\},$ then $N$ cannot be foliated by closed
 planes.
 \end{remark}

 \vspace{.5cm}
 \section{Classifying open 3-manifolds foliated by planes whose fundamental group
  is Abelian of rank at least two}

 In this section we prove the Theorems \ref{teo:subgrabelian} and \ref{teo:abeliano}.
 We state  some results in $3$-manifolds topology and foliations by planes
 which will be used in the proof of the Theorems \ref{teo:subgrabelian}
 and \ref{teo:abeliano}. We begin with some terminology. An three
 manifold $N$ is called \textit{irreducible} if every embedded two
 sphere in $N$ bounds a three ball in $N.$ A two sided surface $S$
 in $N$ is \textit{incompressible} if the map $\pi_1(S)\to \pi_1(N)$
 induced by the inclusion is injective.
 The following results are well-know tools in three manifold topology.

 \begin{theorem}{\cite[Theorem 6]{rosenberg}}
 \label{teo:rosenberg}
 Let $N$ be an orientable $3$-manifold, not necessarily compact, and $\mathcal{F}$
 be a foliation of $N$ by planes. Then  $N$ is irreducible.
 \end{theorem}

 Gabai, in \cite[Proof of Corollary 8.6]{gabai}, proved the
 following result.

%

 \begin{theorem}
 \label{teo:gabaiNaberto}
 If $N$ is an open orientable irreducible $3$-manifold and $\mathbb{Z}\times\mathbb{Z}$
 is a subgroup of $\pi_1(N)$, then $N$ contains an incompressible torus.
 \end{theorem}

 \subsection{Proof of Theorem \ref{teo:subgrabelian}}
 The following result is crucial for to show the Theorem
 \ref{teo:subgrabelian}.

 \begin{proposition}
 \label{prop:abeliansubgroup}
 Let $\mathcal{F}$ be a foliation by planes of an open connected $n$-manifold
 $N$. If $H\cong \mathbb{Z}^k$ is a subgroup of $\pi_1(N)$, then $k<n$.
 \end{proposition}
 To prove this proposition we need the following classical result in algebraic
 topology, see for example \cite{hatcher}.

 \begin{lemma}
 \label{lem:topalg}
 Let $X$ be a normal space, $F$ a closed subspace of $X$ and $Y$
 an ENR space. Suppose that $f_0:F\to Y$ and $g:X\to Y$ are continuous
 maps. If $f_0$ is homotopic to $g_0=g|_F,$ then there exists a continuous map
 $f:X\to Y$ such that $f|_F=f_0$ and $f$ is homotopic to $g.$
 \end{lemma}

 Let $X$ be a topological space and $G$ a group. We say that $X$ is
 a \textit{$K(G,1)$-space} if $\pi_1(X)=G$ and $\pi_p(X)=\{1\}, p\geq 2.$

 \begin{proof}[Proof of Proposition \ref{prop:abeliansubgroup}]
 Firstly we claim that
 \begin{itemize}
   \item[(a)] \textit{$N$ is a $K(\pi_1(N),1)$-space}.
 \end{itemize}

 In fact, let $\widetilde{N}$ be universal covering of $N$. It follows,  by Theorem \ref{teo:Palmeira},
 that $\widetilde{N}=\mathbb{R}^n$. Consequently, $\pi_p(\widetilde{N})=\{1\},$ for all
 $p=1,2,\dots,n$.
 Consider the exact sequence of homotopy groups
 $$
 \cdots \to \pi_p(F)\to \pi_p(\widetilde{N})\to \pi_p(N)\to
 \pi_{p-1}(F)\to \cdots
 $$
 where $F=\pi_1(N)$ (the fiber). Since $\pi_p(F)=\{1\},\ \,p\geq 1,$
 we obtain that $\pi_p(\widetilde{N})= \pi_p(N)$ for all $p\geq 1.$
 This proves our claim.

 Now, we have two possibilities.
 \begin{itemize}
    \item[(b)]\textit{Case $H=\pi_1(N)$:}
 \end{itemize}

 By contradiction, we suppose that $\pi_1(N)\cong \mathbb{Z}^n.$
    Since $N$ and $\mathbb{T}^n$ are $K(\mathbb{Z}^n,1)$-spaces,
    we obtains that $N$ and $\mathbb{T}^n$ have the same
    homotopic type. Consequently, by a classical theorem of
    algebraic topology (Whitehead's Theorem), $H_p(N)=H_p(\mathbb{T}^n)$
    for all $p=1,\dots,n.$ But, as $N$ is open and connected,
    $H_n(N)=0$. This contradicts the fact that
    $H_n(N)=H_n(\mathbb{T}^n)=\mathbb{Z}$
    and, therefore proves that $k<n.$

  \begin{itemize}
  \item[(c)] \textit{General case:}
  \end{itemize}

  Let us consider $p:\widehat{N}\to
    N$ be the covering map associated to $H$,
    that is, $\pi_1(\widehat{N})=H.$ Let $\widehat{\mathcal{F}}$ be the foliations of
    $\widehat{N}$ induced by $\mathcal{F}$, this foliation is by planes.
    Thus, by the first case we obtain $k<n$.
 \end{proof}

 \vspace{0.5cm}
 \begin{proof}[Proof of Theorem \ref{teo:subgrabelian}]
 By Proposition \ref{prop:abeliansubgroup}, every Abelian subgroup
 of $\pi_1(N)$ is isomorphic to $\mathbb{Z}$ or to $\mathbb{Z}^2.$ Consequently the fundamental
  group of $N$ is a free product of the form
 $$
 \pi_1(N)=H_1\ast \cdots \ast H_{\ell}\,,
 $$
 where $H_j,\ j=1,\dots, \ell,$ are subgroups of $\pi_1(N)$ isomorphics
 to $\mathbb{Z}$ or $\mathbb{Z}^2,$
 respectively.

 We suppose that $j\in \{1,\dots, \ell\}$ is such that
 $H_j$ is isomorphic to $\mathbb{Z}^2.$
 By Theorems \ref{teo:rosenberg} and \ref{teo:gabaiNaberto}, there exists an incompressible
 2-torus $T_j$ contained in $N$ such that $\pi_1(T_j)=H_j$. Deforming $T_j$, if necessary, by an isotopy
 of identity we can assume that $T_j$ is in general position with respect to $\mathcal{F}.$
 Then, $\mathcal{F}$ induces a foliation $\mathcal{G}_j$ on $T_j$ having finitely many
 singularities each of which is locally topologically equivalent either to a center or
 to a saddle point of a vector field.
 We can assume that the foliation $\mathcal{G}_j$ is defined by a vector
  field $G_j$. Certainly, we may assume that $T_j$ has been chosen so
  that no pair of singularities of $G_j$ is into the same leaf of $\mathcal{F},$
  in other words, $G_j$ has no saddle connections.

  \begin{itemize}
    \item[(a)] \textit{$G_j$ is topologically equivalent to irrational
  flow.}
  \end{itemize}

  In fact, by using Rosenberg's arguments (see \cite[pag. 137]{rosenberg}),
  we may deform $T_i$ by an isotopy of identity so that $G_j$ has no
  singularities. Hence, since $\mathcal{F}$ is by planes and $T_j$ is incompressible, $G_j$
  has no closed orbit. Therefore, $G_j$ is topologically equivalent to irrational
  flow.

  \vspace{.5cm}
  Let $N_j$ be the saturated by $\mathcal{F}$ of $T_j$. Clearly, $N_j$ is a open
  3-submanifold of $N$ invariant by $\mathcal{F}$). We will show that
  $N_j$ is homeomorphic to $\mathbb{T}^2\times \mathbb{R}.$
  Let $\mathcal{F}_j$ be the restriction of $\mathcal{F}$
  at $N_j$ and $p:T\mathcal{F}_j\to N_j$ be the vector bundle such that for all $x\in N_j$
  the fiber $p^{-1}(x)$ is tangent at $x$ to leaf of $\mathcal{F}_j$
  (which is also a leaf of $\mathcal{F}$) passing by $x.$
  By deforming $T_j$, if necessary, we can assume that:

  \begin{itemize}
    \item[(b)] \textit{If $\theta:T_j\to TN_j$ is the normal vector field to $T_j$ in $N_j$, then
    $\theta$ is tangent to $\mathcal{F}$, that is, $\theta: T_j \to p^{-1}(T_j)$.
    }
  \end{itemize}

  By (a) of the proof of Proposition \ref{prop:abeliansubgroup}, we have that $N_j$ is a
  $K(\pi_1(N_j),1)$-space. Then, by Whitehead's theorem, the inclusion $i:T_j\to N_j$
  is a homotopy equivalence. Consequently:
  \begin{itemize}
    \item[(c)]\textit{there exists a continuous map $f:N_j\to T_j$
  such that $f\circ i$ is homotopic to $id_{T_j}$ and $i\circ f$ is homotopic to $id_{N_j}$}
  \end{itemize}

 We claim that
 \begin{itemize}
   \item[(d)] \textit{there exists a retraction $r:N_j\to T_j$ such that $r|_{T_j}=id_{T_j}$
   and $r$ is homotopic to $f.$}
 \end{itemize}
 Indeed, calling $X=N_j=Y,\ F=T_j,\ f_0=f|_{T_j}$ e $g=f:N_j\to N_j$,
 by Lemma \ref{lem:topalg}, there exists $r:N_j\to T_j$ such that
 $r|_{T_j}=id_{T_j}$ and $r$ is homotopic to $f.$

  \begin{itemize}
    \item[(e)] \textit{$N_j$ is homeomorphic to $\mathbb{T}^2\times \mathbb{R}$.}
  \end{itemize}
  The pull-back of the vector bundle $p^{-1}(T_j)\to T_j$ by the retraction $r:N_i\to T_i$ given in (d)
  is exactly the vector bundle $T\mathcal{F}_j\to N_j$ since
  $r^{\ast}(p^{-1}(T_j))=T\mathcal{F}_j$. Consequently, the vector
  field $\Theta=r^{\ast}\theta$ is tangent to $\mathcal{F}_j$ and is
  normal to $T_j.$ This implies that $N_j$ is homeomorphic to
  $\mathbb{T}^2\times \mathbb{R}.$

  Finally, if for $i\neq j$ the group $H_i$ is isomorphic to $\mathbb{Z}^2$, then
  we obtain another open 3-submanifold of $N$ which is invariant by $\mathcal{F}$
  and homeomorphic to $\mathbb{T}^2\times \mathbb{R}.$ Furthermore, $N_i\cap N_j=\emptyset$.
 \end{proof}

 \subsection{Proof of Theorem \ref{teo:abeliano}}
 \vspace{.5cm}
 \begin{proof}[Proof of Theorem \ref{teo:abeliano}]
 By Theorem \ref{teo:subgrabelian}, there exists $N_0$
 an open submanifold of $N$ homeomorphic to $\mathbb{T}^2\times \mathbb{R}$
 which is invariant by $\mathcal{F}$ such that $(\mathcal{F}|_{N_0},N_0)$
 is topologically equivalent to $(\mathcal{F}_0 \times \mathbb{R}, \mathbb{T}^2\times \mathbb{R})$
 where $\mathcal{F}_0$ is the foliation on $\mathbb{T}^2$ defined
 by the irrational flow. Note that $\pi_1(N_0)=\pi_1(N)\cong \mathbb{Z}^2$

  Let $B$ be a connected component of
  $N\setminus \overline{N_0}$. Then
  \begin{itemize}
  \item[(a)]\textit{$B$ is homeomorphic to $\mathbb{R}^3$ and every leaf in $B$ is closed.}
  \end{itemize}

   Firstly we claim that $\pi_1(B)= \{1\}$. In fact,
   by using the Van Kampen Theorem, we have
   that $\pi_1(B)\ast \pi_1(N_0)$ is subgroup of $\pi_1(N)$ and hence, since $\pi_1(N_0)=\pi_1(N)$,
   we have that  $\pi_1(B)= \{1\}$. Now, we are going to show that
   every leaf in $B$ is closed.
   We suppose, by contradiction, that there exists a non closed leaf. By classical arguments in the theory
   of foliations, there exists a simple closed curve $\gamma \subset B$ which is transversal  to $\mathcal{F}.$
   Moreover, $\gamma$ is not nullhomotopic in $B$, otherwise, by using Heafliger's arguments,
   we obtain a loop $\alpha$ in a leaf of $\mathcal{F}$ which is non trivial in this leaf contradicting
   the fact that all the leaves are planes. Therefore $\gamma$ is not
   nullhomotopic in $B$, but this contradicts that $\pi_1(B)= \{1\}$. Thus
   every leaf in $B$ is closed. Thence, by Palmeira's
   Theorem, we have that $B$ is homeomorphic to $\mathbb{R}^3$ and $(\mathcal{F}|_B,B)$
   is topologically equivalent to $(\mathcal{F}_b\times \mathbb{R},\mathbb{R}^2\times \mathbb{R})$
   where $\mathcal{F}_b$ is an foliation by lines on $\mathbb{R}^2.$

   Finally by using the Van Kampen Theorem it follows that every
   leaf in $B$ separates $N.$


 \end{proof}

\end{document}